\documentclass{amsart}
\usepackage{amsfonts}

\theoremstyle{plain}
\newtheorem{theorem}{Theorem}
\newtheorem{lemma}{Lemma}

\newtheorem{proposition}{Proposition}
\newtheorem{conjecture}{Conjecture}

\theoremstyle{definition}

\newtheorem{problem}{Problem}

\theoremstyle{remark}
\newtheorem{remark}{Remark}

\numberwithin{equation}{section}
\input{tcilatex}

\begin{document}
\title[Upwards Askey--Wilson scheme]{On peculiar properties of generating
functions of some orthogonal polynomials}
\author{Pawe\l\ J. Szab\l owski}
\address{Department of Mathematics and Information Sciences,\\
Warsaw University of Technology\\
pl. Politechniki 1, 00-661 Warsaw, Poland}
\email{pawel.szablowski@gmail.com}
\date{March 11, 2012}
\subjclass{Primary 33D45, 26C05; Secondary 05A30}
\keywords{$q-$Hermite, Al-Salam--Chihara, continuous dual Hahn,
Askey--Wilson polynomials, Askey--Wilson scheme, generating functions,
Laguerre., Chebyshev polynomials.}
\thanks{The author is grateful to an unknown referee for drawing his
attention to Rainville's book and pointing out numerous misprints.}

\begin{abstract}
We prove that for $\left\vert x\right\vert \leq 1,\left\vert t\right\vert
<1\allowbreak ,$ $-1\allowbreak <q\allowbreak \leq 1$ and $n\geq 0:$ $%
\allowbreak \sum_{i\geq 0}\frac{t^{i}}{(q)_{i}}h_{n+i}\left( x|q\right)
\allowbreak =\allowbreak h_{n}\left( x|t,q\right) \allowbreak \allowbreak
\sum_{i\geq 0}\frac{t^{i}}{(q)_{i}}h_{i}\left( x|q\right) $, where $%
h_{n}\left( x|q\right) $ and $h_{n}\left( x|t,q\right) $ are respectively
the so called $q-$Hermite and the big $q-$Hermite polynomials and $(q)_{n}$
denotes the so called $q-$Pochhammer symbol. We prove similar equalities
involving big $q-$Hermite and Al-Salam--Chihara polynomials and
Al-Salam--Chihara and the so called continuous dual $q-$Hahn polynomials.
Moreover we are able to relate in this way some other 'ordinary ' orthogonal
polynomials such as e.g. Hermite, Chebyshev or Laguerre. These equalities
give new interpretation of the polynomials involved and moreover can give
rise to a simple method of generating more and more general (i.e. involving
more and more parameters) families of orthogonal polynomials.

We pose some conjectures concerning Askey--Wilson polynomials and their
possible generalizations. We prove that these conjectures are true for the
cases $q\allowbreak =\allowbreak 1$ (classical case) and $q\allowbreak
=\allowbreak 0$ (free case) thus paving the way to generalization of
Askey--Wilson polynomials at least in these two cases.
\end{abstract}

\maketitle

\section{Introduction and auxiliary results}

\subsection{Introduction}

In the paper we play around with the scheme:%
\begin{equation}
\sum_{i\geq 0}\frac{t^{i}}{(q)_{i}}p_{n+i}\left( x\right) \allowbreak
=\allowbreak q_{n}\left( x|t\right) \sum_{i\geq 0}\frac{t^{i}}{\left(
q\right) _{i}}p_{i}\left( x\right) ,  \label{*_*}
\end{equation}%
where the series on both sides converge for $x,t$ from certain Cartesian
product of bounded intervals and $\left\{ p_{i}\right\} _{i\geq 0}$ and $%
\left\{ q_{i}\right\} _{i\geq 0}$ are certain families of orthogonal
polynomials. So far we were able to relate in this scheme the $q-$Hermite
and the big $q-$Hermite polynomials, the big $q-$Hermite and the
Al-Salam--Chihara polynomials and also the Al-Salam-Chihara and the so
called continuous dual $q-$Hahn polynomials and $(q)_{n}$ denotes the so
called $q-$Pochhammer symbol.

In fact the idea of considering `shifted generating function' like the left
hand side of (\ref{*_*}) and relate it to the same generating function
without the shift (i.e. when $n\allowbreak =\allowbreak 0)$ is not new and
appeared in a version confined to two families of orthogonal polynomials
(Hermite and Laguerre) in the book of Rainville \cite{Rainville}.

In this paper we treat it as the general idea, show that it is useful and
give more examples. These examples concern particularly polynomials
satisfying the so called Askey-Wilson scheme. Moreover we treat it as the
tool to `move upwards' the Askey-Wilson scheme and in particular as a tool
towards possible generalization of the Askey--Wilson polynomials.

As we already pointed out above this scheme can be applied also to some
`classical' orthogonal polynomials as Rainville did in his book \cite%
{Rainville} after some necessary modification

Can we continue this scheme and interpret e.g. in this way Askey--Wilson
polynomials? Can we `go beyond' Askey--Wilson polynomials? From what we have
been able to prove so far it seems that we have a simple scheme to produce
families of orthogonal polynomials with more and more parameters.
Askey--Wilson polynomials provide the largest family (in the sense of the
number of parameters) of orthogonal polynomials of one variable that has
been relatively well described. Thus now there is a chance to `go beyond'
these polynomials. Of course it requires further research and cannot be
settled down in one small size article.

We interpret so obtained relationships between the above mentioned
polynomials with the help of ordinary $q-$difference operator just acting in
the spirit of recent attempt to move upwards the Askey-Wilson scheme done by
Atakishiyeva\&Atakishiyev in \cite{At-At10} but with much simpler operators
and what is more important indicating the way to `climb upwards this scheme'
beyond Askey-Wilson polynomials.

The paper is organized as follows. In the next two subsections we provide
simple introduction to $q-$series theory presenting typical notation used
and presenting a few typical families of the so called basic orthogonal
polynomials. The word basic comes from the base which is a parameter in most
cases denoted by $q$ and such that $-1<q\leq 1$. Then in Section \ref{glow}
we present our main results, open questions and remarks are in Section \ref%
{open} while less interesting laborious proofs are in Section \ref{dowody}.

\subsection{Notation}

We use notation traditionally used in the so called $q-$series theory. Since
not all readers are familiar with it we will recall now this notation.

$q$ is a parameter such that $-1<q\leq 1$ unless otherwise stated. Let us
define $\left[ 0\right] _{q}\allowbreak =\allowbreak 0;$ $\left[ n\right]
_{q}\allowbreak =\allowbreak 1+q+\ldots +q^{n-1}\allowbreak ,$ $\left[ n%
\right] _{q}!\allowbreak =\allowbreak \prod_{j=1}^{n}\left[ j\right] _{q},$
with $\left[ 0\right] _{q}!\allowbreak =1$ and%
\begin{equation*}
\QATOPD[ ] {n}{k}_{q}\allowbreak =\allowbreak \left\{ 
\begin{array}{ccc}
\frac{\left[ n\right] _{q}!}{\left[ n-k\right] _{q}!\left[ k\right] _{q}!} & 
, & n\geq k\geq 0 \\ 
0 & , & otherwise%
\end{array}%
\right. .
\end{equation*}%
It will be useful to use the so called $q-$Pochhammer symbol for $n\geq 1:$%
\begin{eqnarray*}
\left( a;q\right) _{n} &=&\prod_{j=0}^{n-1}\left( 1-aq^{j}\right) , \\
\left( a_{1},a_{2},\ldots ,a_{k};q\right) _{n}\allowbreak &=&\allowbreak
\prod_{j=1}^{k}\left( a_{j};q\right) _{n}.
\end{eqnarray*}%
with $\left( a;q\right) _{0}=1$. Often $\left( a;q\right) _{n}$ as well as $%
\left( a_{1},a_{2},\ldots ,a_{k};q\right) _{n}$ will be abbreviated to $%
\left( a\right) _{n}$ and $\left( a_{1},a_{2},\ldots ,a_{k}\right) _{n},$ if
it will not cause misunderstanding.

It is easy to notice that $\left( q\right) _{n}=\left( 1-q\right) ^{n}\left[
n\right] _{q}!$ and that\newline
$\QATOPD[ ] {n}{k}_{q}\allowbreak =$\allowbreak $\allowbreak \left\{ 
\begin{array}{ccc}
\frac{\left( q\right) _{n}}{\left( q\right) _{n-k}\left( q\right) _{k}} & ,
& n\geq k\geq 0 \\ 
0 & , & otherwise%
\end{array}%
\right. $. \newline
Notice that $\left[ n\right] _{1}\allowbreak =\allowbreak n,\left[ n\right]
_{1}!\allowbreak =\allowbreak n!,$ $\QATOPD[ ] {n}{k}_{1}\allowbreak
=\allowbreak \binom{n}{k},$ $\left( a;1\right) _{n}\allowbreak =\allowbreak
\left( 1-a\right) ^{n}$ and $\left[ n\right] _{0}\allowbreak =\allowbreak
\left\{ 
\begin{array}{ccc}
1 & if & n\geq 1 \\ 
0 & if & n=0%
\end{array}%
\right. ,$ $\left[ n\right] _{0}!\allowbreak =\allowbreak 1,$ $\QATOPD[ ] {n%
}{k}_{0}\allowbreak =\allowbreak 1,$ $\left( a;0\right) _{n}\allowbreak
=\allowbreak \left\{ 
\begin{array}{ccc}
1 & if & n=0 \\ 
1-a & if & n\geq 1%
\end{array}%
\right. .$

\subsection{Auxiliary notions and results}

\subsubsection{Orthogonal polynomials}

Following \cite{IA}, or \cite{Koek} we will define the following families of
polynomials.

The $q-$ Hermite (briefly qH) polynomials denoted by $\left\{ h_{n}\left(
x|q\right) \right\} _{n\geq 0}$ constitute the one parameter family of
orthogonal polynomials satisfying the following 3-term recurrence:%
\begin{equation}
h_{n+1}\left( x|q\right) \allowbreak =\allowbreak 2xh_{n}\left( x|q\right)
-(1-q^{n})h_{n-1}\left( x|q\right) ,  \label{_H}
\end{equation}%
with $h_{-1}\left( x|q\right) \allowbreak =\allowbreak 0$ and $h_{0}\left(
x|q\right) \allowbreak =\allowbreak 1.$

The the big $q-$ Hermite polynomials (briefly bqH) denoted by $\left\{
h_{n}\left( x|a,q\right) \right\} _{n\geq -1}$ constitute the $2-$parameter
family of orthogonal polynomials that satisfy the following 3-term
recurrence:%
\begin{equation}
h_{n+1}\left( x|a,q\right) \allowbreak =\allowbreak (2x-aq^{n})h_{n}\left(
x|a,q\right) -(1-q^{n})h_{n-1}\left( x|a,q\right) ,  \label{_bH}
\end{equation}%
with $h_{-1}\left( x|a,q\right) \allowbreak =\allowbreak 0,$ $h_{0}\left(
x|a,q\right) \allowbreak =\allowbreak 1$.

The Al-Salam--Chihara polynomials (briefly ASC) denoted by $\left\{
Q_{n}\left( x|a,b,q\right) \right\} _{n\geq -1}$ constitute the $3-$%
parameter family of orthogonal polynomials that satisfy the following 3-term
recurrence:%
\begin{equation}
Q_{n+1}\left( x|a,b,q\right) =(2x-(a+b)q^{n})Q_{n}\left( x|a,b,q\right)
-(1-abq^{n-1})(1-q^{n})Q_{n-1}\left( x|y,\rho ,q\right) ,  \label{_Q}
\end{equation}%
with $\allowbreak Q_{-1}\left( x|a,b,q\right) \allowbreak =\allowbreak 0,$ $%
Q_{0}\left( x|a,b,q\right) \allowbreak =\allowbreak 1$. \newline
For $n\allowbreak =\allowbreak 0$ we set $(1-abq^{n-1})(1-q^{n})Q_{n-1}%
\left( x|y,\rho ,q\right) $ to $0.$

The continuous dual Hahn polynomials (briefly c2h) denoted by $\left\{ \psi
_{n}\left( x|a,b,c,q\right) \right\} _{n\geq -1}$ constitute the $4-$%
parameter family of orthogonal polynomials that satisfy the following 3-term
recurrence: 
\begin{equation}
\psi _{n+1}\left( x|a,b,c,q\right) \allowbreak =\allowbreak (2x-d_{n})\psi
_{n}\left( x|a,b,c,q\right) -f_{n-1}(1-q^{n})\psi _{n-1}\left(
x|a,b,c,q\right) ,  \label{_c2h}
\end{equation}%
with $\allowbreak \psi _{-1}\left( x|a,b,c,q\right) \allowbreak =0,$ $\psi
_{0}\left( x|a,b,c,q\right) \allowbreak =1$ and coefficients $d_{n}$ and $%
f_{n}$ given by for $n\geq 0:$ 
\begin{eqnarray*}
d_{n}\allowbreak &=&\allowbreak (a+b+c)q^{n}+abcq^{n-1}(1-q^{n}-q^{n+1}), \\
f_{n}\allowbreak &=&\allowbreak (1-abq^{n})(1-acq^{n})(1-bcq^{n}).
\end{eqnarray*}%
Again for $n\allowbreak =\allowbreak 0$ we set $f_{n-1}(1-q^{n})\psi
_{n-1}\left( x|a,b,c,q\right) \allowbreak =\allowbreak 0.$

Orthogonality of all these polynomials takes place on $[-1,1].$ From
Favard's theorem it follows that if for all $n>0:\allowbreak
(1-abq^{n-1})\allowbreak \geq \allowbreak 0$ and $%
(1-abq^{n})(1-acq^{n})(1-bcq^{n})\allowbreak \geq \allowbreak 0$ then
respectively ASC and c2h polynomials are orthogonal with respect to a
positive measure.

Let us also mention the so called Askey--Wilson polynomials (briefly AW) $%
\left\{ AW_{n}\right\} _{n\geq -1}$. They can be defined with the help of
the basic hypergeometric function as it was done in the original paper \cite%
{AW85} of Askey and Wilson or by the 3-term recurrence as done in \cite{Koek}
or \cite{IA}. The polynomials that we will call Askey-Wilson and denote $%
AW_{n}$ are in fact equal to $2^{n}p_{n}\left( x\right) $ where polynomials $%
p_{n}$ are defined by 3-term recurrence (3.1.5) in \cite{Koek}. For our
purpose it will be enough to define them in the following way:%
\begin{equation}
AW_{n}(x|a,b,c,d,q)\allowbreak =\allowbreak \sum_{i=0}^{n}\QATOPD[ ] {n}{i}%
_{q}\left( -a\right) ^{n-i}q^{\binom{n-i}{2}}\frac{\left(
bcq^{i},bdq^{i},cdq^{i}\right) _{n-i}}{\left( abcdq^{n+i-1}\right) _{n-i}}%
\psi _{i}\left( x|b,c,d,q\right) ,  \label{AW}
\end{equation}%
as proved in \cite{Szab-bAW} formula (2.5). Again we understand that $%
AW_{-1}(x|a,b,c,d,q)\allowbreak =\allowbreak 0$ and $AW_{0}(x|a,b,c,d,q)%
\allowbreak =\allowbreak 1.$

Finally let us mention Chebyshev polynomials $\left\{ U_{n}\right\} _{n\geq
-1}$ of the second kind that satisfy the following 3-term recurrence:%
\begin{equation}
2xU_{n}\left( x\right) =U_{n+1}\left( x\right) +U_{n-1}\left( x\right) ,
\label{_cheb}
\end{equation}%
with $U_{-1}\left( x\right) \allowbreak =\allowbreak 0$ and $U_{0}\left(
x\right) \allowbreak =\allowbreak 1.$ These polynomials will play an
auxiliary r\^{o}le.

\subsubsection{Properties of some orthogonal polynomials}

We have the following elementary Lemma:

\begin{lemma}
\label{lin_komb}Suppose $a_{i}\left( x\right) \allowbreak =\allowbreak
\sum_{k=0}^{s}\beta _{k}U_{i-k}\left( x\right) $ for $i\allowbreak
=\allowbreak 1,2,\ldots ,n$ for some constants $\beta _{j},$ $j\allowbreak
=\allowbreak 0,\ldots ,s$ with $n\geq s.$ Suppose also that for $m\geq n$ we
have 
\begin{equation*}
a_{m+1}\left( x\right) \allowbreak =\allowbreak 2xa_{m}\left( x\right)
-a_{m-1}\left( x\right) ,
\end{equation*}%
then $\forall m\geq n$ : 
\begin{equation*}
a_{m}\left( x\right) \allowbreak =\allowbreak \sum_{k=0}^{s}\beta
_{k}U_{m-k}\left( x\right) .
\end{equation*}
\end{lemma}

\begin{proof}
The proof is by induction. For $m\allowbreak =\allowbreak n$ it is true by
assumption. Let us assume that it is true for $m\allowbreak =\allowbreak j.$
Then for $m\allowbreak =\allowbreak j+1$ we have $a_{j+1}\left( x\right)
\allowbreak =\allowbreak 2x\sum_{k=0}^{s}\beta _{k}U_{j-k}\left( x\right)
\allowbreak -\allowbreak \sum_{k=0}^{s}\beta _{k}U_{j-1-k}\left( x\right)
\allowbreak =\allowbreak \sum_{k=0}^{s}\beta _{k}(U_{j+1-k}\left( x\right)
+U_{j-1-k}\left( x\right) )-\allowbreak \sum_{k=0}^{s}\beta
_{k}U_{j-1-k}\left( x\right) \allowbreak =\allowbreak $

$\sum_{k=0}^{s}\beta _{k}U_{j+1-k}\left( x\right) .$
\end{proof}

As an immediate corollary we have the following Proposition with some
assertions already known. We present them here together in order to expose
the regularities and pave the way to possible generalizations. This
proposition will help to justify the conjecture concerning generalization of
our main result that will be presented below in Section \ref{open}.

\begin{proposition}
\label{q->0}i) $\forall n\geq 0:$ 
\begin{equation*}
h_{n}\left( x|0\right) \allowbreak =U_{n}\left( x\right) ,\allowbreak
h_{n}\left( x|a,0\right) \allowbreak =\allowbreak U_{n}\left( x\right)
\allowbreak -\allowbreak aU_{n-1}\left( x\right) ,
\end{equation*}

ii) $\forall n\geq 1:$%
\begin{equation*}
Q_{n}(x|a,b,0)\allowbreak =\allowbreak U_{n}(x)\allowbreak -\allowbreak
(a+b)U_{n-1}(x)+abU_{n-2}(x),
\end{equation*}

iii) $\forall n\geq 1:$%
\begin{equation*}
\psi _{n}\left( x|a,b,c,0\right) \allowbreak =\allowbreak
U_{n}(x)\allowbreak -\allowbreak (a+b+c)U_{n-1}(x)\allowbreak +\allowbreak
(ab\allowbreak +\allowbreak cb\allowbreak +\allowbreak
ac)U_{n-2}(x)\allowbreak -\allowbreak abcU_{n-3}\left( x\right) ,
\end{equation*}

iv) $\forall n\geq 2:$%
\begin{gather*}
AW_{n}\left( x|a,b,c,d,0\right) \allowbreak =\allowbreak U_{n}\left(
x\right) \allowbreak -\allowbreak (a+b+c+d)U_{n-1}\left( x\right) \allowbreak
\\
+\allowbreak (ab+ac+ad+bc+bd+cd)U_{n-2}\left( x\right) \allowbreak
-\allowbreak (abc+abd+bcd+acd)U_{n-3}\allowbreak +\allowbreak
abcdU_{n-4}\left( x\right) ,
\end{gather*}%
where $AW_{n}$ denotes Askey--Wilson polynomial as defined by (\ref{AW}).
\end{proposition}

\begin{proof}
Is shifted to Section \ref{dowody}.
\end{proof}

\begin{remark}
\label{zera}Notice that $\forall n\geq -1$ 
\begin{gather*}
h_{n}\left( x|0\right) =U_{n}\left( x\right) ,~h_{n}\left( x|0,q\right)
\allowbreak =\allowbreak h_{n}\left( x|q\right) ,~Q_{n}\left( x|a,0,q\right)
\allowbreak =\allowbreak h_{n}\left( x|a,q\right) , \\
~\psi _{n+1}\left( x|a,b,0,q\right) =Q_{n}\left( x|a,b,q\right) .
\end{gather*}
\end{remark}

\begin{remark}
\label{intuition} To support intuition let us remark following e.g. \cite%
{Szablowski2010(1)} that $\lim_{q\rightarrow 1^{-}}\allowbreak h_{n}\left( x%
\frac{\sqrt{1-q}}{2}|a\sqrt{1-q},q\right) \allowbreak /\allowbreak
(1-q)^{n/2}\allowbreak =\allowbreak He_{n}\left( x-a\right) $, where $He_{n}$
denotes the $n-th$ so called `probabilistic' Hermite polynomial i.e.
polynomial orthogonal with respect to measure with the density $\exp
(-x^{2}/2)/\sqrt{2\pi }.$
\end{remark}

For completeness of the exposition let us mention that polynomials $q-$%
Hermite, bqH, ASC are related mutually by the following relationships: (see
e.g. \cite{Szablowski2010(3)} and \cite{At-At11}(4.9) )%
\begin{eqnarray}
h_{n}\left( x|a,q\right) \allowbreak &=&\sum_{k=0}^{n}\QATOPD[ ] {n}{k}%
_{q}(-a)^{k}q^{\binom{k}{2}}h_{n-k}\left( x|q\right) ,  \label{hnaha} \\
Q_{n}\left( x|a,b,q\right) \allowbreak &=&\allowbreak \sum_{k=0}^{n}\QATOPD[ 
] {n}{k}_{q}(-a)^{k}q^{\binom{k}{2}}h_{n-m}\left( x|b,q\right) .
\label{anah}
\end{eqnarray}%
Recently the c2h polynomials were related to the ASC polynomials by the
following relationship (after slight modification of \cite{At-At11}(2.7)):%
\begin{equation}
\psi _{n}\left( x|a,b,c,q\right) \allowbreak =\allowbreak \sum_{i=0}^{n}%
\QATOPD[ ] {n}{i}_{q}\left( -a\right) ^{i}q^{\binom{i}{2}%
}(bcq^{n-i})_{i}Q_{n-i}\left( x|b,c,q\right) .  \label{psinaQ}
\end{equation}

\subsection{General Result}

We end up this section by the presentation of an auxiliary simple result
that will be used several times in the sequel. We have the following
Proposition.

\begin{proposition}
\label{pomoc} Let $\sigma _{n}\allowbreak \left( \rho ,q\right) =\allowbreak
\sum_{i\geq 0}\frac{\rho ^{i}}{\left( q\right) _{i}}\xi _{n+i}$ for $%
\left\vert \rho \right\vert <1$ and certain sequence $\left\{ \xi
_{m}\right\} _{m\geq 0}$ such that $\sigma _{n}$ exists for every $n.$ Then 
\begin{equation}
\sigma _{n}\left( \rho q^{m},q\right) \allowbreak =\allowbreak
\sum_{k=0}^{m}\left( -1\right) ^{k}\QATOPD[ ] {m}{k}_{q}q^{\binom{k}{2}}\rho
^{k}\sigma _{n+k}(\rho ,q).  \label{sigmanm}
\end{equation}
\end{proposition}

\begin{proof}
An easy, not very interesting proof by induction is shifted to Section \ref%
{dowody}.
\end{proof}

\section{Main Results\label{glow}}

\begin{theorem}
\label{basic} i) For $\forall n\geq 0;\allowbreak x^{2}\leq 1;t^{2}<1:$ 
\begin{equation}
\sum_{i\geq 0}\frac{t^{i}}{\left( q\right) _{i}}h_{i+n}\left( x|q\right)
\allowbreak =\allowbreak h_{n}\left( x|t,q\right) \sum_{i\geq 0}\frac{t^{i}}{%
\left( q\right) _{i}}h_{i}\left( x|q\right) ,  \label{_sbh}
\end{equation}%
where $h_{n}\left( x|t,q\right) $ is the bqH polynomial defined by (\ref{_bH}%
).

ii) For $\forall n\geq 0;\allowbreak x^{2}\leq 1;\allowbreak
t^{2}<1\allowbreak ;\left\vert at\right\vert <1:$ 
\begin{equation}
\sum_{i\geq 0}\frac{t^{i}}{\left( q\right) _{i}}h_{i+n}\left( x|a,q\right)
\allowbreak =\allowbreak \frac{Q_{n}\left( x|a,t,q\right) }{\left( at\right)
_{n}}\sum_{i\geq 0}\frac{t^{i}}{\left( q\right) _{i}}h_{i}\left(
x|a,q\right) ,  \label{_sASC}
\end{equation}%
where $Q_{n}\left( x|a,t,q\right) $ is the ASC polynomial defined by (\ref%
{_Q}).

iii) For $\forall n\geq 0;\allowbreak \left\vert x\right\vert \leq
1;\allowbreak t^{2}<1\allowbreak ;\left\vert at\right\vert ,\left\vert
bt\right\vert <1:$ 
\begin{equation}
\sum_{i\geq 0}\frac{t^{i}}{\left( q\right) _{i}}Q_{i+n}\left( x|a,b,q\right)
\allowbreak =\allowbreak \frac{\psi _{n}\left( x|a,b,t,q\right) }{\left(
at,bt\right) _{n}}\sum_{i\geq 0}\frac{t^{i}}{\left( q\right) _{i}}%
Q_{i}\left( x|a,b,q\right) ,  \label{_sc2h}
\end{equation}%
where polynomial $\psi _{n}(x|a,b,t,q)$ is the c2h polynomials defined by (%
\ref{_c2h}).
\end{theorem}

\begin{proof}
Is shifted to section \ref{dowody}.
\end{proof}

It turns out that similar properties can be attributed to some other,
classical orthogonal polynomials. In the case of Hermite and Laguerre
polynomials this was observed by Rainville in his book \cite{Rainville}. Let 
$\left\{ H_{n}\left( x\right) \right\} _{n\geq -1}$ and $\left\{ U_{n}\left(
x\right) \right\} _{n\geq -1}$ denote classical, orthogonal, polynomials
respectively Hermite and Chebyshev. Let us consider polynomials $\left\{
\lambda _{n}\left( x,\alpha \right) \right\} _{n\geq -1}$ that are monic
versions of Laguerre polynomials. More precisely 
\begin{equation*}
\lambda _{n}\left( x,\alpha \right) =(-1)^{n}n!L_{n}^{\left( \alpha \right)
}(x),
\end{equation*}%
where $L_{n}^{\left( \alpha \right) }(x)$ denote traditional Laguerre
polynomials e.g. defined by (1.11.1) of \cite{Koek}. The 3-term recurrences
of polynomials $\left\{ H_{n}\left( x\right) \right\} _{n\geq -1},$ $\left\{
\lambda _{n}\left( x,\alpha \right) \right\} _{n\geq -1}$ and $\left\{
U_{n}\left( x\right) \right\} _{n\geq -1}$ are given by formulae
respectively (1.13.4), (1.11.4) of \cite{Koek} and (\ref{_cheb}).

\begin{lemma}
\label{classical}For $\forall n\geq 0$, $x,t\in \mathbb{R}$ with $t\neq -1$
in the case of (\ref{f_L}) we have. 
\begin{eqnarray}
\sum_{j\geq 0}\frac{t^{j}}{j!}H_{n+j}\left( x\right) &=&H_{n}\left(
x-t\right) \sum_{j\geq 0}\frac{t^{j}}{j!}H_{j}\left( x\right) ,  \label{f_H}
\\
\sum_{j\geq 0}\frac{t^{j}}{j!}\lambda _{n+j}\left( x,\alpha \right) &=&\frac{%
\lambda _{n}\left( \frac{x}{1+t},\alpha \right) }{(1+t)^{n}}\sum_{j\geq 0}%
\frac{t^{j}}{j!}\lambda _{j}\left( x,\alpha \right) ,  \label{f_L} \\
\sum_{j\geq 0}t^{j}U_{n+j}\left( x\right) \allowbreak &=&\allowbreak
(U_{n}\left( x\right) -tU_{n-1}\left( x\right) )\sum_{j\geq
0}t^{j}U_{j}\left( x\right) .  \label{f_U}
\end{eqnarray}
\end{lemma}

\begin{proof}
(\ref{f_H}) is proved in \cite{Rainville} p. 197 eq. (1), (\ref{f_L}) is
proved in \cite{Rainville} p. 211 eq (9) with an obvious modification such
at the change of $t$ to $-t.$ Proof of (\ref{f_U}) is the following. Let $%
\nu _{n}\left( x,t\right) \allowbreak =\allowbreak \sum_{j\geq
0}t^{j}U_{n+j}\left( x\right) $. We have 
\begin{equation*}
2x\nu _{n}\left( x,t\right) \allowbreak =\allowbreak \sum_{j\geq
0}t^{j}(U_{n+j+1}\left( x\right) +U_{n+j-1}\left( x\right) )\allowbreak
=\allowbreak \nu _{n+1}\left( x,t\right) +\nu _{n-1}\left( x,t\right) ,
\end{equation*}%
with 
\begin{eqnarray*}
\nu _{1}\left( x,t\right) \allowbreak &=&\allowbreak \sum_{j\geq
0}t^{j}U_{1+j}\left( x\right) \allowbreak =\allowbreak \sum_{j\geq
0}t^{j}(2xU_{j}\left( x\right) -U_{j-1}\left( x\right) )\allowbreak
=\allowbreak 2x\nu _{0}\left( x,t\right) -t\nu _{0}\left( x,t\right)
\allowbreak \\
&=&\allowbreak (2x-t)\nu _{0}\left( x,t\right) \allowbreak =\allowbreak
(U_{1}\left( x\right) -tU_{0}\left( x\right) )\nu _{0}\left( x,t\right) .
\end{eqnarray*}

Hence by Proposition \ref{lin_komb} we have: $\nu _{n}\left( x,t\right) /\nu
_{0}\left( x,t\right) \allowbreak =\allowbreak (U_{n}\left( x\right)
-tU_{n-1}\left( x\right) ).$
\end{proof}

\section{Open problems and comments\label{open}}

\begin{remark}
\label{diff}Recalling $q-$differentiation formula (see e.g. \cite{IA}%
(11.4.1) on page 296) $D_{q,x}f\left( x\right) =\frac{f(qx)-f(x)}{x(q-1)}$
we see that 
\begin{equation}
D_{q,\rho }(\frac{\rho ^{n}}{\left( q\right) _{n}})\allowbreak =\allowbreak 
\frac{\rho ^{n-1}}{(1-q)\left( q\right) _{n-1}}.  \label{rozn}
\end{equation}%
Let us define $n-$fold composition of the operator $D.$ 
\begin{equation*}
D_{q,x}^{n}f(x)\allowbreak =\allowbreak D_{q,x}(D_{q,x}(...D_{q,x}f(x))).
\end{equation*}%
Hence we deduce that expressions of the form $\varphi _{n}\left(
x|t,q\right) \allowbreak =\allowbreak \sum_{i\geq 0}\frac{t^{i}}{\left(
q\right) _{i}}p_{i+n}\left( x|q\right) $ considered in the first three
assertions of Theorem \ref{basic} are in fact following (\ref{rozn})
proportional to $q-$derivatives of $\varphi _{0}\left( x|t,q\right) $ with
respect to $t.$ Hence those assertions can be expressed in the following `$%
q- $Rodrigues'-like form:%
\begin{eqnarray*}
D_{q,t}^{n}\left( \sum_{i\geq 0}\frac{t^{i}}{\left( q\right) _{i}}%
h_{i}\left( x|q\right) \right) \allowbreak &=&\allowbreak
(1-q)^{n}h_{n}\left( x|t,q\right) \sum_{i\geq 0}\frac{t^{i}}{\left( q\right)
_{i}}h_{i}\left( x|q\right) , \\
D_{q,t}^{n}\left( \sum_{i\geq 0}\frac{t^{i}}{\left( q\right) _{i}}%
h_{i}\left( x|a,q\right) \right) \allowbreak &=&\allowbreak \frac{%
(1-q)^{n}Q_{n}\left( x|a,t,q\right) }{\left( at\right) _{n}}\sum_{i\geq 0}%
\frac{t^{i}}{\left( q\right) _{i}}h_{i}\left( x|a,q\right) , \\
D_{q,t}^{n}\left( \sum_{i\geq 0}\frac{t^{i}}{\left( q\right) _{i}}%
Q_{i}\left( x|a,b,q\right) \right) \allowbreak &=&\allowbreak \frac{%
(1-q)^{n}\psi _{n}\left( x|a,b,t,q\right) }{\left( at,bt\right) _{n}}%
\sum_{i\geq 0}\frac{t^{i}}{\left( q\right) _{i}}Q_{i}\left( x|a,b,q\right) .
\end{eqnarray*}
\end{remark}

Thus it is natural to pose the following hypothesis.

\begin{conjecture}
\label{first}%
\begin{equation*}
D_{q,t}^{n}\left( \sum_{i\geq 0}\frac{t^{i}}{(q)_{i}}\psi _{i}\left(
x|a,b,c,q\right) \right) \allowbreak =\allowbreak \frac{(1-q)^{n}AW_{n}%
\left( x|a,b,c,t,q\right) }{(at,bt,ct)_{n}}\sum_{i\geq 0}\frac{t^{i}}{(q)_{i}%
}\psi _{i}\left( x|a,b,c,q\right) ,
\end{equation*}%
where $AW_{n}$ denotes $n-$th. Askey--Wilson polynomial as defined in (\ref%
{AW}).
\end{conjecture}

\begin{remark}
Let us notice that the above mentioned conjecture is almost trivially true
for $q\allowbreak =\allowbreak 1$ in view of (\ref{f_H}) with an obvious
modification that all polynomial families considered are modified in the
following way.\ Instead of polynomials $h_{n}$ we consider polynomials $%
\allowbreak H_{n}\left( x|q\right) =\allowbreak h_{n}\left( x\frac{\sqrt{1-q}%
}{2}|q\right) /(1-q)^{n/2}$. Instead of polynomials $h_{n}\left(
x|a,q\right) $ we consider polynomials $H_{n}\left( x|a,q\right) \allowbreak
=\allowbreak h_{n}\left( x\frac{\sqrt{1-q}}{2}|a\sqrt{1-a},q\right)
/(1-q)^{n/2}.$ Instead of polynomials $Q_{n}\left( x|a,b,q\right) $ we
consider polynomials $P_{n}\left( x|a,b,q\right) \allowbreak =\allowbreak
Q_{n}\left( x\frac{\sqrt{1-q}}{2}|a\sqrt{1-a},b\sqrt{1-q},q\right)
/(1-q)^{n/2}$. Instead of polynomials $\psi _{n}\left( x|a,b,c,q\right) $ we
consider polynomials $G_{n}\left( x|a,b,c,q\right) \allowbreak =\allowbreak
\psi _{n}\left( x\frac{\sqrt{1-q}}{2}|a\sqrt{1-a},b\sqrt{1-q},c\sqrt{1-q}%
,q\right) /(1-q)^{n/2}$. Finally instead of polynomials $AW_{n}\left(
x|a,b,c,d,q\right) $ we consider polynomials \newline
$a_{n}\left( x|a,b,c,d,q\right) \allowbreak =\allowbreak AW_{n}\left( x\frac{%
\sqrt{1-q}}{2}|a\sqrt{1-a},b\sqrt{1-q},c\sqrt{1-q},d\sqrt{1-q},q\right)
/(1-q)^{n/2}$ while all generating functions involved are defined as the sum 
$\sum_{i\geq 0}\frac{t^{i}}{\left[ i\right] _{q}!}p_{n+i}\left( x\right) $
where instead of $p_{n}$ we put $H_{n}\left( x|q\right) $ or $H_{n}\left(
x|a,q\right) $ or $P_{n}\left( x|a,b,q\right) $ or $G_{n}\left(
x|a,b,a,q\right) .$ Then for $q\allowbreak =\allowbreak 1$ we have $%
H_{n}\left( x|1\right) \allowbreak =\allowbreak He_{n}\left( x\right) ,$ $%
H_{n}\left( x|a,1\right) \allowbreak =\allowbreak He_{n}\left( x-a\right) ,$ 
$P_{n}\left( x|a,b,1\right) \allowbreak =\allowbreak He_{n}\left(
x-a-b\right) ,$ $G_{n}\left( x|a,b,c,1\right) \allowbreak =\allowbreak
He_{n}\left( x-a-b-c\right) ,$ \newline
$a_{n}\left( x|a,b,a,d,1\right) \allowbreak =\allowbreak He_{n}\left(
x-a-b-c-d\right) ,$ where $He_{n}\left( x\right) $ denotes, as before, the
probabilistic Hermite polynomial as described in Remark \ref{intuition}.
\end{remark}

\begin{proposition}
\label{Con dla q=0}Conjecture \ref{first} is true for $q\allowbreak
=\allowbreak 0.$
\end{proposition}

\begin{proof}
Proof is shifted to Section \ref{dowody}.
\end{proof}

\begin{remark}
A successful attempt to get different polynomials from the Askey--Wilson
scheme was made in \cite{At-At10}. It was done through modification of the
so called Askey--Wilson divided $q-$difference operator, a complicated $q-$%
differentiation scheme, applied straightforwardly to polynomials themselves
as well as to the densities of measures that make these polynomials
orthogonal. Our differentiation scheme presented in the Remark \ref{diff}
involves much simpler $q-$difference operator applied not directly to the
polynomials involved but to the characteristic functions of these
polynomials.
\end{remark}

\begin{problem}[Open Problem]
Continuing the line of generalizations presented in Theorem \ref{basic} and
Conjecture \ref{first} one can pose the following question. Is it true that:%
\begin{equation*}
D_{q,t}^{n}\left( \sum_{i\geq 0}\frac{t^{i}}{(q)_{i}}AW_{i}\left(
x|a,b,c,d,q\right) \right) =GAW_{n}\left( x|a,b,c,d,t,q\right) \sum_{i\geq 0}%
\frac{t^{i}}{(q)_{i}}AW_{i}\left( x|a,b,c,d,q\right) ,
\end{equation*}%
where $\left\{ GAW_{n}\left( x|a,b,c,d,t,q\right) \right\} _{n\geq -1}$
constitute a family of orthogonal (i.e. satisfying some 3-term recurrence)
polynomials in $x.$

If it was true then naturally polynomials $GAW_{n}$ could be regarded as
generalization of the AW polynomials. Then of course one could pose many
further questions concerning properties of these polynomials. One of such
questions could be to compare this attempt to generalize AW polynomials with
another one presented in \cite{Szabl-intAW}.
\end{problem}

\section{Proofs\label{dowody}}

\begin{proof}[Proof of Proposition \protect\ref{q->0}]
Assertions i) ii) iii) and assertion iv) for complex parameters were
mentioned already in \cite{Szablowski2010(1)}, however in view of Lemma \ref%
{lin_komb} the proof can be reduced to the following two arguments. First of
all notice that examining equations (\ref{_H}), (\ref{_bH}), (\ref{_Q}), (%
\ref{_c2h}) and Proposition 2 of \cite{Szab-bAW} that for $n\geq 2$ for the
first three assertions and for $n\geq 3$ for the AW polynomials for $%
q\allowbreak =\allowbreak 0$ the 3-term recurrences satisfied by all
mentioned in the proposition polynomials are in fact identical with (\ref%
{_cheb}). Now it remains to check (by direct computation done for example
with the help of package Mathematica) that indeed for $n\allowbreak
=\allowbreak 0,1,2,3$ all polynomials in question have the form mentioned in
Proposition \ref{q->0}. Another justification of all assertion follow
formulae (\ref{anah}), (\ref{hnaha}), (\ref{psinaQ}) and (\ref{AW}) from
which it follows for $q\allowbreak =\allowbreak 0$ that for $n\geq 1:$ $%
AW_{n}(x|a,b,c,d,0)\allowbreak =\allowbreak \psi _{n}\left( x|b,c,d,0\right)
-a\psi _{n-1}\left( x|b,c,d,0\right) $ and similarly for other polynomials
considered. One has to be cautious with the case $n\allowbreak =\allowbreak
1 $ and AW polynomials. One can easily check that for $n\allowbreak
=\allowbreak 1$ assertion iv) is not true.
\end{proof}

\begin{proof}[Proof of Proposition \protect\ref{pomoc}]
First we will prove that%
\begin{equation}
\sigma _{n}\left( \rho q^{m},q\right) \allowbreak =\allowbreak \sigma
_{n}\left( \rho q^{m-1},q\right) -\rho q^{m-1}\sigma _{n+1}\left( \rho
q^{m-1},q\right) .  \label{sigman}
\end{equation}%
We have: 
\begin{eqnarray*}
\sigma _{n}\allowbreak \left( \rho q^{m},q\right)  &=&\allowbreak
\sum_{i\geq 0}\frac{q^{mi}\rho ^{i}}{\left( q\right) _{i}}\xi
_{n+i}\allowbreak =\allowbreak \sum_{i\geq 0}\frac{q^{(m-1)i}\rho ^{i}}{%
\left( q\right) _{i}}\xi _{n+i}-\sum_{i\geq 0}\frac{q^{(m-1)i}(1-q^{i})\rho
^{i}}{\left( q\right) _{i}}\xi _{n+i} \\
&=&\sigma _{n}\left( \rho q^{m-1},q\right) -\rho q^{m-1}\sum_{j\geq 0}\frac{%
q^{(m-1)j}t^{j}}{\left( q\right) _{j}}\xi _{n+1+j}.
\end{eqnarray*}%
Then we prove (\ref{sigmanm}) by induction with respect to $m.$ We see that
it is true for $m\allowbreak =\allowbreak 1.$ Hence let us assume that it is
true for $m\leq k.$ Let us consider $m\allowbreak =\allowbreak k+1.$ We
have: 
\begin{eqnarray*}
&&\sigma _{n}\left( \rho q^{k+1},q\right) =\sigma _{n}\left( \rho
q^{k},q\right) -\rho q^{k}\sigma _{n+1}\left( \rho q^{k},q\right)  \\
&=&\sum_{j=0}^{k}(-1)^{j}\QATOPD[ ] {k}{j}_{q}q^{\binom{j}{2}}\rho
^{j}\sigma _{n+j}\left( \rho ,q\right) -\rho q^{k}\sum_{j=0}^{k}(-1)^{j}%
\QATOPD[ ] {k}{j}_{q}q^{\binom{j}{2}}\rho ^{j}\sigma _{n+1+j}\left( \rho
,q\right)  \\
&=&\sigma _{n}\left( \rho ,q\right) +\left( -1\right) ^{k+1}\rho ^{k+1}q^{%
\binom{k+1}{2}}\sigma _{n+k+1}\left( \rho ,q\right) +\sum_{j=1}^{k}(-1)^{j}%
\QATOPD[ ] {k}{j}_{q}q^{\binom{j}{2}}\rho ^{j}\sigma _{n+j}\left( \rho
,q\right)  \\
&&+\sum_{j=0}^{k-1}\left( -1\right) ^{j+1}\QATOPD[ ] {k}{j}_{q}q^{k+\binom{%
j+1}{2}-j}(1-q)^{j+1}\rho ^{j+1}\sigma _{n+j+1}\left( \rho ,q\right) .
\end{eqnarray*}%
Now we change the index of summation from $j=s-1.$ and get:%
\begin{eqnarray*}
&&\sigma _{n}\left( \rho q^{k},q\right) -\rho q^{k}\sigma _{n+1}\left( \rho
q^{k},q\right)  \\
&=&\sigma _{n}\left( \rho ,q\right) +\left( -1\right) ^{k+1}\rho ^{k+1}q^{%
\binom{k+1}{2}}\sigma _{n+k+1}\left( \rho ,q\right) +\sum_{j=1}^{k}\left(
-1\right) ^{j}(\QATOPD[ ] {k}{j}_{q} \\
&&+q^{k-j+1}\QATOPD[ ] {k}{j-1}_{q})q^{\binom{j}{2}}\rho ^{j}\sigma
_{n+j}\left( \rho ,q\right) .
\end{eqnarray*}%
since $\binom{j-1}{2}\allowbreak +\allowbreak j-1\allowbreak =\allowbreak 
\binom{j}{2}.$ Now we use the fact that $\QATOPD[ ] {k}{j}_{q}+q^{k-j+1}%
\QATOPD[ ] {k}{j-1}_{q}\allowbreak =\allowbreak \QATOPD[ ] {k+1}{j}_{q}$.
Hence we see that (\ref{sigmanm}) is true.
\end{proof}

\begin{proof}[Proof of Teorem \protect\ref{basic}]
Notice that i) and ii) follow iii) by Remark \ref{zera}. Thus it is enough
to prove iii).

iii) Let us denote $\chi _{n}\left( x|a,b,t,q\right) \allowbreak
=\allowbreak \sum_{i\geq 0}\frac{t^{i}}{\left( q\right) _{i}}Q_{i+n}\left(
x|a,b,q\right) .$ We are using (\ref{sigmanm}) on the way. 
\begin{gather*}
2x\chi _{n}\left( x|a,b,t,q\right) \allowbreak =\sum_{i\geq 0}\frac{t^{i}}{%
\left( q\right) _{i}}(2x-(a+b)q^{n+i}+(a+b)q^{n+i})Q_{n+i}\left(
x|a,b,q\right) \allowbreak  \\
=(a+b)q^{n}\chi _{n}\left( x|a,b,tq,q\right) \allowbreak +\allowbreak
\sum_{i\geq 0}\frac{t^{i}}{\left( q\right) _{i}}(Q_{n+1+i}\left(
x|a,b,q\right) \allowbreak  \\
+\allowbreak (1-abq^{n+i-1})(1-q^{n+i})Q_{n+i-1}\left( x|a,b,q\right) ) \\
=(a+b)q^{n}\chi _{n}\left( x|a,b,t,q\right) \allowbreak -\allowbreak
(a+b)q^{n}t\chi _{n+1}\left( x|a,b,t,q\right) \allowbreak +\allowbreak \chi
_{n+1}\left( x|a,b,t,q\right) \allowbreak + \\
\allowbreak (1-q^{n})\sum_{i\geq 0}\frac{t^{i}}{\left( q\right) _{i}}%
(1-abq^{n+i-1})Q_{n+i-1}\left( x|a,b,q\right) \allowbreak  \\
+\allowbreak q^{n}\sum_{i\geq 0}\frac{t^{i}}{\left( q\right) _{i}}\left(
1-q^{i}\right) (1-abq^{n+i-1})Q_{n+i-1}\left( x|a,b,q\right) .\allowbreak
\allowbreak \allowbreak 
\end{gather*}%
Further we have%
\begin{gather*}
2x\chi _{n}\left( x|a,b,t,q\right) \allowbreak =(a+b)q^{n}\chi _{n}\left(
x|a,b,q,q\right) \allowbreak +\allowbreak (1-(a+b)q^{n}t)\chi _{n+1}\left(
x|a,b,t,q\right) \allowbreak  \\
+\allowbreak (1-q^{n})(1-abq^{n-1})\chi _{n-1}\left( x|a,b,t,q\right)
\allowbreak  \\
+\allowbreak (1-q^{n})abq^{n-1}\sum_{i\geq 0}\frac{t^{i}}{\left( q\right)
_{i}}(1-q^{i})Q_{n+i-1}\left( x|a,b,q\right) \allowbreak  \\
+\allowbreak q^{n}(1-abq^{n})\sum_{i\geq 0}\frac{t^{i}}{(q)_{i}}%
(1-q^{i})Q_{n+i-1}\left( x|a,b,q\right) \allowbreak  \\
+\allowbreak q^{n}abq^{n}\sum_{i\geq 0}\frac{t^{i}}{(q)_{i}}%
(1-q^{i})(1-q^{i-1})Q_{n+i-1}\left( x|a,b,q\right) \allowbreak  \\
=\chi _{n+1}\left( x|a,b,t,q\right) (1-(a+b)q^{n}t+q^{2n}abt^{2})\allowbreak
+ \\
\chi _{n}\left( x|a,b,t,q\right)
((a+b)q^{n}+(1-q^{n})abq^{n-1}t+q^{n}(1-abq^{n})t)+ \\
\allowbreak (1-q^{n})(1-abq^{n-1})\chi _{n-1}\left( x|a,b,t,q\right) .
\end{gather*}%
Hence we have the following equation:%
\begin{eqnarray}
&&(2x-(a+b+t)q^{n}-abtq^{n-1}(1-q^{n}-q^{n+1}))\chi _{n}\left(
x|a,b,t,q\right) \allowbreak \allowbreak   \label{(***)} \\
&=&(1-(a+b)q^{n}t+q^{2n}abt^{2})\chi _{n+1}\left( x|a,b,t,q\right)
\allowbreak   \notag \\
&&+\allowbreak (1-q^{n})(1-abq^{n-1})\chi _{n-1}\left( x|a,b,t,q\right) . 
\notag
\end{eqnarray}
Notice that $1-(a+b)q^{n}t+q^{2n}abt^{2})\allowbreak =\allowbreak \left(
1-atq^{n}\right) \left( 1-btq^{n}\right) .$

Let $\hat{\chi}_{n}\left( x|a,b,t,q\right) \allowbreak =\allowbreak \chi
_{n}\left( x|a,b,t,q\right) \left( at,bt\right) _{n}.$ We have then after
multiplying both sides of (\ref{(***)}) by $\left( at,bt\right) _{n}$ 
\begin{eqnarray}
&&\hat{\chi}_{n}\left( x|a,b,t,q\right)
(2x-(a+b+t)q^{n}-abtq^{n-1}(1-q^{n}-q^{n+1}))\allowbreak   \label{3tr_chi} \\
&=&\allowbreak \hat{\chi}_{n+1}\left( x|a,b,t,q\right)
+(1-q^{n})(1-abq^{n-1}\left( 1-atq^{n-1}\right)   \notag \\
&&\times \left( 1-btq^{n-1}\right) \hat{\chi}_{n-1}\left( x|a,b,t,q\right) .
\notag
\end{eqnarray}
with $\hat{\chi}_{0}\left( x|a,b,t,q\right) \allowbreak =\allowbreak \chi
_{0}\left( x|a,b,t,q\right) \allowbreak \allowbreak .$ Hence $\hat{\chi}_{n}$
satisfies the same 3-term recurrence as $\psi _{n}$ compare (\ref{_c2h}).
Besides we have 
\begin{gather*}
\chi _{1}\left( x|a,b,t,q\right) \allowbreak =\allowbreak \sum_{i\geq 0}%
\frac{t^{i}}{(q)_{i}}Q_{1+i}\left( x|a,b,q\right)  \\
=\sum_{i\geq 0}\frac{t^{i}}{\left[ i\right] _{q}!}((2x-(a+b)q^{i})Q_{i}%
\left( x|a,b,q\right) \allowbreak \allowbreak -\allowbreak
(1-q^{i})_{q}\left( 1-abq^{i-1}\right) Q_{i-1}\left( x|a,b,q\right)  \\
=(2x-t)\chi _{0}\left( x|a,b,t,q\right) -\left( a+b\right) \chi _{0}\left(
x|a,b,tq,q\right) +tab\chi _{0}\left( x|a,b,tq,q\right) \allowbreak  \\
=\allowbreak (2x-a-b-t+tab)\chi _{0}\left( x|a,b,t,q\right)  \\
+(a+b)t\chi _{1}\left( x|a,b,t,q\right) \allowbreak -\allowbreak t^{2}ab\chi
_{1}\left( x|a,b,t,q\right) .
\end{gather*}%
So:%
\begin{equation*}
\chi _{1}\left( x|a,b,t,q\right) \allowbreak =\allowbreak \frac{%
(2x-a-b-t+tab)}{(1-(a+b)t+t^{2}ab)}\chi _{0}\left( x|a,b,t,q\right) .
\end{equation*}%
Consequently $\hat{\chi}_{1}\left( x|a,b\,,t,q\right) \allowbreak
=\allowbreak (2x-a-b-t+tab)\hat{\chi}_{0}\left( x|a,b,t,q\right) $ and we
deduce that $\hat{\chi}_{-1}\left( x|a,b,t,q\right) \allowbreak =\allowbreak
0.$ Thus we deduce examining equation (\ref{3tr_chi}) that: $\hat{\chi}%
_{n}\left( x|a,b,t,q\right) /\hat{\chi}_{0}\left( x|a,b,t,q\right) $
satisfies 3-term recurrence the same as the one satisfied by continuous dual 
$q-$Hahn polynomials, with the same initial conditions.
\end{proof}

\begin{proof}[Proof of Proposition \protect\ref{Con dla q=0}]
By Proposition \ref{q->0}, iii) we have 
\begin{gather*}
\sum_{j\geq 0}t^{j}\psi _{j+n}\left( x|a,b,c,0\right) \allowbreak
=\allowbreak \\
\sum_{j\geq 0}t^{j}(U_{n+j}\left( x\right) -(a+b+c)U_{n+j-1}\left( x\right)
\allowbreak +\allowbreak (ab+ac+bc)U_{n+j-2}\left( x\right)
-abcU_{n+j-3}\left( x\right) ).
\end{gather*}%
Let us denote $\chi _{n}(x|a,b,c,t)\allowbreak =\allowbreak \sum_{j\geq
0}t^{j}\psi _{j+n}\left( x|a,b,c,0\right) .$ Hence:%
\begin{equation*}
\chi _{n}(x|a,b,c,t)\allowbreak =\allowbreak \alpha _{n}(x,t)-(a+b+c)\alpha
_{n-1}(x,t)+(ab+ac+bc)\alpha _{n-2}(x,t)-abc\alpha _{n-3}(x,t),
\end{equation*}%
where we denoted $\alpha _{n}\left( x,t\right) \allowbreak =\allowbreak
\sum_{j\geq 0}t^{j}U_{j+n}\left( x\right) $. $\alpha _{n}(x,t)$ was already
calculated in Lemma \ref{classical}, (\ref{f_U}) and is equal to $%
\allowbreak \allowbreak \alpha _{0}(x,t)(U_{n}\left( x\right)
-tU_{n-1}\left( x\right) ).\allowbreak $

Hence 
\begin{gather*}
\chi _{n}\left( x|a,b,c,t\right) \allowbreak =\allowbreak (U_{n}\left(
x\right) \allowbreak -\allowbreak tU_{n-1}\left( x\right) \allowbreak
-\allowbreak (a+b+c)U_{n-1}\allowbreak +\allowbreak
(at+bt+ct)U_{n-2}\allowbreak \\
+\allowbreak (ab+ac+bc)U_{n-2}\allowbreak -\allowbreak
(abt+act+bct)U_{n-3}\allowbreak -\allowbreak abcU_{n-3}\allowbreak
+\allowbreak abctU_{n-4}\allowbreak )\alpha _{0}\left( x,t\right)
\allowbreak = \\
(U_{n}\left( x\right) \allowbreak -\allowbreak (a+b+c+t)U_{n-1}\left(
x\right) \allowbreak +\allowbreak (ab+ac+bc+at+bt+ct)U_{n-2}\left( x\right)
\allowbreak \\
-\allowbreak (abc+tab+tac+tbc)U_{n-3}(x)+abctU_{n-4}\left( x\right) )\alpha
_{0}\left( x,t\right) \allowbreak =\allowbreak \allowbreak \allowbreak
AW_{n}\left( x|a,b,c,t,0\right) \alpha _{0}\left( x,t\right) ,
\end{gather*}%
by Proposition \ref{q->0}, iv). On the other hand we have: 
\begin{gather*}
\chi _{0}(x|a,b,c,t)=\sum_{j\geq 0}t^{j}(U_{j}\left( x\right)
-(a+b+c)U_{j-1}\left( x\right) \allowbreak +\allowbreak
(ab+ac+bc)U_{j-2}\left( x\right) -abcU_{j-3}\left( x\right) )\allowbreak = \\
\allowbreak \alpha _{0}\left( x,t\right) \allowbreak (1-\allowbreak
(a+b+c)t+(ab+ac+bc)t^{2}-abct^{3})\allowbreak =\alpha _{0}\left( x,t\right)
\allowbreak (1-at)(1-bt)(1-ct).
\end{gather*}%
So 
\begin{equation*}
\chi _{n}\left( x|a,b,c,t\right) \allowbreak =\allowbreak AW_{n}\left(
x|a,b,c,t,0\right) \frac{\chi _{0}\left( x|a,b,c,t\right) }{%
(1-at)(1-bt)(1-ct)}.
\end{equation*}
\end{proof}

\end{document}